\documentclass[english, a4, 12pt]{amsart}

\usepackage{multicol} 
\usepackage{babel}
\usepackage{amstext}
\usepackage{amsmath}
\usepackage{amsthm}
\usepackage{amsfonts}
\usepackage{latexsym}
\usepackage{ifthen}
\usepackage{xypic}
\usepackage{amssymb}

\xyoption{all}
\newcommand{\Rien}[1]{}

\newcommand{\F}{{\mathcal F}}

\renewcommand{\O}{{\mathcal O}}

\newcommand{\PN}{{\mathbb P}}

\newcommand{\Pic}{{\rm Pic}}

\newcommand{\lra}{\longrightarrow}
\newcommand{\KC}{{\mathbb C}}
\newcommand{\KZ}{{\mathbb Z}}
\newcommand{\KQ}{{\mathbb Q}}

\newcommand{\KN}{{\mathbb N}}

\newcommand{\Sieg}{\mathfrak H}

\newcommand{\KR}{\mathbb R}
\newcommand{\sO}{{\mathcal O}}

\newcommand{\Ref}[1]{(\ref{#1})}

\newtheorem{lemma1}[equation]{}

\newenvironment{lemma}{\begin{lemma1}{\bf Lemma.}}{\end{lemma1}}

\newenvironment{example}{\begin{lemma1}{\bf Example.}\rm}{\end{lemma1}}
\newenvironment{abs}{\begin{lemma1}\rm}{\end{lemma1}}
\newenvironment{theorem}{\begin{lemma1}{\bf Theorem.}}{\end{lemma1}}

\newenvironment{proposition}{\begin{lemma1}{\bf Proposition.}}{\end{lemma1}}
\newenvironment{corollary}{\begin{lemma1}{\bf Corollary.}}{\end{lemma1}}

\newenvironment{definition}{\begin{lemma1}{\bf Definition.}}{\end{lemma1}}

\begin{document}

\title{Semistability of restricted tangent bundles \\
and a question of I. Biswas}
\author[P. Jahnke]{Priska Jahnke}
\address{Priska Jahnke - Mathematisches Institut - Freie Universit\"at Berlin
  - Arnimallee 3 - D-14195 Berlin, Germany}
\email{priska.jahnke@fu-berlin.de}
\author[I.Radloff]{Ivo Radloff}
\address{Ivo Radloff - Mathematisches Institut - Universit\"at Bayreuth -
  D-95440 Bayreuth, Germany}
\email{ivo.radloff@uni-bayreuth.de}
\date{\today}
\maketitle

\section*{Introduction}

Let $X$ be a complex projective manifold, $n = \dim_{\KC} X \ge 2$. Denote by $T_X$ its holomorphic tangent bundle. In many cases, $T_X$ is $H$--semistable with respect to some ample $H \in \Pic(X)$, for example when $X$ is K\"ahler--Einstein. If $T_X$ is $H$--semistable, then $T_X$ is semistable when restricted to a general complete intersection curve $C$ cut out by $n-1$ general hyperplanes in $|mH|$ for $m \gg 0$ by a theorem of Mehta and Ramanathan. It is therefore not unlikely for $T_X$ to be semistable on a {\em general} curve. 

In \cite{Biswas}, I. Biswas raised the following question: suppose $\nu^*T_X$ is semistable for {\em every} holomorphic $\nu: C \lra X$ from a compact Riemann surface. Does this mean $X$ is a finite \'etale quotient of an abelian variety? Biswas proved that the answer is yes in the case of projective surfaces as well as in the K\"ahler--Einstein case. Using different methods we will prove the general case:

\begin{theorem} \label{Main}
  On a projective manifold of dimension $n \ge 2$ the following two conditions are equivalent.
\begin{enumerate}
  \item $\nu^*T_X$ is semistable for every holomorphic $\nu: C \lra X$ from a compact Riemann surface $C$.
  \item $X$ is a finite \'etale quotient of an abelian variety.
\end{enumerate}
\end{theorem}

Our strategy is as follows. If $X$ satisfies 1.), $T_X$ is holomorphically projectively flat, i.e., $\PN(T_X)$ comes from a representation of the fundamental group
  \[\sigma: \pi_1(X) \lra PGl_{n-1}(\KC).\]
(\cite{Biswas} or \Ref{PrFl}). We first prove that $X$ is free of rational curves \Ref{RatCurve}. Moreover $K_X$ is nef but not big \Ref{KE2}. We distinguish two cases: $\sigma(\pi_1(X))$ finite or infinite. The finite case is what we expect and here $X$ is indeed a finite \'etale quotient of an abelian variety \Ref{Finite}. The second possibility has to be excluded.

In the infinite case, ideas of Koll\'ar around the Shafarevich conjecture and of Zuo concerning representations of K\"ahler groups show that $X$ admits a rational dominant map onto some smooth variety of general type whose general fibers are good minimal models of Kodaira dimension $\kappa = 0$ or $1$ \Ref{Fibr}. We conclude that $K_X$ is in fact abundant and obtain the Iitaka fibration $f: X \lra Y$ \Ref{Abundance}. Again from a result of Koll\'ar we infer that, perhaps after some finite \'etale cover, $X$ is  an abelian group scheme over a base of general type \Ref{AbSch}. Estimates on the positivity of Hodge bundles associated to abelian fibrations are due to Viehweg and Zuo and will give the final contradiction \Ref{Arakelov}.

\section{Preliminaries}
\setcounter{equation}{0}
In this section we recall basic facts concerning nef-ness and semistability. Let $X$ be a projective manifold, let $E$ be a rank $r > 0$ holomorphic vector bundle. 

\begin{abs}{\bf Nef.} (\cite{DPS}) A line bundle $L \in \Pic(X)$ is called {\em nef} if $\deg \nu^*L \ge 0$ for every holomorphic map $\nu: C \lra X$ from a compact Riemann surface $C$. The vector bundle $E$ is called {\em nef} if $\sO_{\PN(E)}(1)$ on $\PN(E)$ is nef. It is called {\em numerically flat}, if $E$ and $E^*$ are nef.
\end{abs}

\begin{abs}{\bf Stability.} (\cite{KobB}) Unless $X$ is a curve this notion depends on a choice of an ample line bundle $H \in \Pic(X)$. For a nonzero torsion free sheaf $\F$ we define the {\em avaraged first class} as
 \[\mu(\F) := \frac{c_1(\F)}{rk \F}.\]
The {\em slope of $\F$ with respect to $H$} is the rational number $\mu_H(\F) := \mu(\F).H^{n-1}$. A vector bundle $E$ of rank $r>0$ is called {\em $H$--stable} (resp. {\em $H$--semistable)}, if for every coherent subsheaf $\F \subset E$ of rank $0 < rk \F < rk E$ we have  
  \[\mu_H(\F) < \mu_H(E) \quad (resp. \;\; \mu_H(\F) \le \mu_H(E)).\]
A theorem due to Mehta and Ramanathan says $E$ is $H$--semistable if and only if the restriction of $E$ to a general complete intersection curve cut out by hypersurface $H_i \in |mH|$, $m \gg 0$, is semistable.
\end{abs}

\begin{abs}{\bf Nef versus semistable.} In general, $\det E$ will not be divisible by $r$ in $\Pic(X)$. We can nevertheless consider $\frac{\det E}{r}$ as a real line bundle. We say that
   \[E \otimes \frac{\det E^*}{r}\]
is {\em nef}, if $S^r E \otimes \det E^*$ is nef. Equivalently, $-K_{\PN(E)/X} = -K_{\PN(E)} + \pi^*K_X$ is nef. As $c_1(S^r E \otimes \det E^*)=0$, also the dual bundle is nef, i.e., $E \otimes \frac{\det E^*}{r}$ is numerically flat in the sense of \cite{DPS}.
\end{abs}

\begin{lemma} \label{Miyak}
  The curve case $\dim X = 1$. Then $E$ is semistable if and only if $E \otimes \frac{\det E^*}{r}$ is nef.
\end{lemma}

For a proof see for example \cite{MP}, Part I, (2.4.). More generally, summarizing results by Narasimhan, Seshadri and many others, we have the following theorem (compare \cite{NakNT}, Theorem A):

\begin{theorem} \label{NakThA}
  Let $E$ be a holomorphic vector bundle of rank $r>0$ on the projective manifold $X$. Let $H \in \Pic(X)$ be ample. The following conditions are equivalent:
  \begin{enumerate}
   \item for every holomorphic $\nu: C \lra X$ from a compact Riemann surface, $\nu^*E$ is semistable.
   \item $E \otimes \frac{\det E^*}{r}$ is nef.
   \item $E$ is $H$--semistable and the Bogomolov inequality is an equality
    \[(r-1)c_1^2(E).H^{n-2} = 2rc_2(E).H^{n-2}.\]
   \item $E$ admits a filtration into subbundles
       \[0 = E_0 \subset E_1 \subset E_2 \subset \cdots \subset E_t = E,\]
     where each $E_i/E_{i-1}$ admits a projectively flat Hermititan structure and $\mu(E_i/E_{i-1}) \equiv \mu(E)$ for every $i$.
  \end{enumerate}
\end{theorem}
Here a vector bundle $E$ admits a projectively flat Hermitian structure if it admits a Hermitian metric $h$ whose curvature tensor $\Theta_h$ has the form $\Theta_h = \omega \cdot Id_E$ for some $(1,1)$--Form $\omega$. Equivalenty, $\PN(E)$ is defined by a representation 
  \[\pi_1(X) \lra PU(r)\]
(\cite{NakNT} with different notation, \cite{KobB}).

\begin{proof}
  The equivalence of 2.), 3.), 4.) is Theorem A of \cite{NakNT}. Condition 2.) is equivalent to $S^rE \otimes \det E^*$ nef. A vector bundle $F$ is nef if and only if $\nu^*F$ is nef for every holomorphic map $\nu: C \lra X$ from a curve $C$. Therefore 1.) and 2.) are equivalent by \Ref{Miyak}.
\end{proof}

The following statement is well known for nef vector bundles:
\begin{corollary} \label{PullBack}
  Let $f: Y \lra X$ be holomorphic. Let $E$ be a rank $r>0$ vector bundle on $X$. 
\begin{enumerate}
 \item If $E$ satisfies one of the conditions in \Ref{NakThA}, then $f^*E$ satisfies the conditions in \Ref{NakThA}. 
 \item The converse holds if $f$ is surjective.
\end{enumerate}
\end{corollary}

We call $E$ {\em projectively flat} if $E$ admits a connection $D$ whose curvature $R_D = \omega Id_E$ for some $2$--form $\omega$. Equivalently, $E$ is projectively flat if $\PN(E)$ is defined by a representation 
  \[\pi_1(X) \lra PGl_{r-1}(\KC)\]
(\cite{KobB}). If $E$ admits a projectively flat Hermitian metric, then it is projectively flat. The converse is false in general, see the example below.

\begin{proposition} \label{PrFl}
  Let $E$ satisfies one of the conditions of \Ref{NakThA}. Then 
   \begin{enumerate}
     \item $E$ is projectively flat.
     \item $c_i(E) = \frac{1}{r^i}{r \choose i} c_1^i(E)$ in $H^{2i}(X, \KR)$.
  \end{enumerate}
\end{proposition}

\begin{proof}
 1.) First assume $\mu(E) = 0$. Then $c_1(E).H^{n-1} = 0$ for every $H \in \Pic(X)$ ample. By 3.) in \Ref{NakThA} $c_2(E).H^{n-2} = 0$. We may consider $E$ as a semistable Higgs bundle with zero Higgs field. By \cite{Sim}, (3.10.), $E$ is flat, i.e., admits a flat connection and is defined by a representation  $\pi_1(X) \lra Gl_r(\KC)$. Then $E$ is projectively flat.

Next assume  $\mu(E) \not= 0$. Let $E_1$ be as in 4.) of \Ref{NakThA}. Then $E \otimes E_1^*$ satisfies the conditions of \Ref{NakThA} and $\mu(E \otimes E_1^*) = 0$. Then $E \otimes E_1^*$ admits a flat connection. The tensor product of this connection with the connection associated to the projectively flat Hermitian metric on $E_1$ (\cite{KobB}, \S 5) gives a projectively flat connection on $E \otimes E_1^* \otimes E_1$. The bundle $E_1^* \otimes E_1$ has a trivial direct summand induced by the identity and the trace map. Then $E$ is a direct summand of $E \otimes E_1^* \otimes E_1$. We obtain a projectively flat connection on $E$.

2.) \cite{KobB}, II, (3.1.).
\end{proof}

\begin{example}
  1.) Let $\KZ^{2} \simeq \Lambda \subset \KC$ be a lattice, $0 \not= \chi \in Hom(\Lambda, \KC)$. Then
  \[\lambda \mapsto \left(\begin{array}{cc}
                       1 & \chi(\lambda) \\
                      0 & 1
                  \end{array}\right) \in Sl_2(\KC)\]
is a representation of $\Lambda$. The associated vector bundle corresponds to the non--split extension of $\O_{\KC/\Lambda}$ by $\O_{\KC/\Lambda}$. It is semistable and nef with numerical trivial determinant. It does not carry a (projectively) flat Hermitian metric.

2.) Let $\Gamma \subset Sl_2(\KR)$ be torsion free. Let $\Gamma$ act on $\Sieg_1 = \{z \in \KC | \Im m z > 0\}$ as a group of Moebius transformations. Assume $C = \Sieg_1/\Gamma$ is compact. The canonical representation $\Gamma \lra Sl_2(\KR)$ defines a flat vector bundle that is not nef. 
\end{example}

\section{Condition $(*)$ and first consequences}
\setcounter{equation}{0}
Let us introduce
\begin{quote}
  {\em Condition (*).} $X$ is a projective manifold of dimension $n \ge 2$ and for every compact Riemann surface and every holomorphic $\nu: C \lra X$, the pull back bundle $\nu^*T_X$ is semistable. 
\end{quote}
We derive first consequence, some can already be found in \cite{Biswas}. From \Ref{NakThA} we infer that if $X$ satisfies $(*)$, then $\PN(T_X)$ is defined by a representation
      \begin{equation} \label{sigm}
        \sigma: \pi_1(X) \lra PGl_{n-1}(\KC).
      \end{equation}
Moreover we have 
      \begin{equation} \label{ChernKl} 
        c_i(X) = \frac{1}{n^i}{n \choose i} c_1^i(X) \quad \mbox{ in } H^{2i}(X, \KR).\end{equation}
In the attempt to classify all $X$ that satisfy $(*)$ we may at any time replace $X$ by some finite \'etale cover by \Ref{PullBack}.

\begin{proposition} \label{RatCurve}
  Let $X$ satisfy $(*)$. Any holomorphic map $f: \PN_1 \lra M$ is constant and   \begin{enumerate}
    \item $K_X$ is nef and $K_Y$ is nef for evry smooth submanifold $Y \subset X$.
    \item any rational map $\xymatrix{Y \ar@{..>}[r] & X}$ from a smooth $Y$ is holomorphic.
   \end{enumerate}  
\end{proposition}

\begin{proof}
  Assume $f: \PN_1 \lra X$ is not constant. Then $f^*T_X$ is semistable. Then $f^*T_X \simeq \O_{\PN_1}(a)^{\oplus n}$ for some $a \in \KZ$. From $T_{\PN_1} \hookrightarrow f^*T_X$ we infer $a>0$. Then $f^*T_X$ is ample. Then $X \simeq \PN_n$ (\cite{MP}, 4.2.). It is well known that $T_{\PN_n}$ is $\O_{\PN_n}(1)$--stable. However, the restriction of $T_{\PN_n}$ to a line $l$ splits as   
  \[T_{\PN_n}|_l \simeq \O_{\PN_1}(1)^{\oplus n-1} \oplus \O_{\PN_1}(2)\]
and is not semistable. This contradicts (*).

1.) $K_X$ is nef by the cone theorem (\cite{KM}). 2.) If we had to blow up $Y$ to make $\xymatrix{Y \ar@{..>}[r] & X}$ holomorphic, then we would find a rational curve in $X$.
\end{proof}

\begin{definition}
  A projective manifold $X$ is {\em almost abelian}, if there exists a finite \'etale cover $A \lra X$ from an abelian variety $A$.
\end{definition}

The following theorem can already be found in \cite{Biswas}. We include a proof for convenience of the reader.

\begin{theorem} \label{KE}
  Let $X$ be a K\"ahler--Einstein manifold that satisfies $(*)$. Then $X$ is almost abelian.
\end{theorem}

\begin{proof}
  Denote by $\omega$ a K\"ahler--Einstein form. The Chen Ogiue inequality says
  \[\int_X n c_1(X)^2 \wedge \omega^{n-2} \le \int_X 2(n+1) c_2(X) \wedge \omega^{n-2}\]
with equality if and only if $X$ is of constant holomorphic sectional curvature$s$. By \Ref{ChernKl} we have $(n-1) c_1^2(X) = 2n c_2(X)$. A simple computation shows $\int_X c_1(X)^2 \wedge \omega^{n-2} \le 0$. As $K_X$ is nef, $\int_X c_1(X)^2 \wedge \omega^{n-2} = 0$. Then also $\int_X c_2(X) \wedge \omega^{n-2} = 0$. Then we have equality. Then $s=0$ and $X$ is almost abelian.
\end{proof}

\begin{corollary} \label{KE2}
  Let $X$ satisfy $(*)$. 
   \begin{enumerate}
   \item $K_X$ is not big. 
   \item If $K_X \equiv 0$, then $X$ is almost abelian. 
  \end{enumerate}
\end{corollary}

\begin{proof}
  1.) Assume $K_X$ is big. By the base point free theorem, $|dK_X|$ is spanned for $d \gg 0$. The Iitaka fibration $f: X \lra Y$ is birational, $dK_X \simeq f^*H$ for some $H \in \Pic(Y)$ ample. Positive dimensional fibers would be covered by rational curves by \cite{Kaw}, Theorem 2. By \Ref{RatCurve}, $f$ is an isomorphism. Then $K_X$ is ample. By the famous theorem of Aubin \cite{Au} and Yau \cite{Yau} $X$ is K\"ahler--Einstein. Then $K_X$ big contradicts \Ref{KE}.


2.) Assume $K_X \equiv 0$. Then $c_1(X) = 0$ in $H^2(X, \KR)$. The theorem of Yau \cite{Yau} says $X$ carries a K\"ahler--Einstein metric. The claim follows from \Ref{KE}. 
\end{proof}

\section{Finite case}
\setcounter{equation}{0}
In this section we assume that the representation map 
  \[\sigma: \pi_1(X) \lra PGl_{n-1}(\KC)\]
defining $\PN(T_X)$ has a finite image. Note that this is the case if $X$ is almost abelian. We prove the converse:

\begin{proposition} \label{Finite}
  Let $X$ satisfy $(*)$. If $\sigma$ has a finite image, then $X$ is almost abelian.
\end{proposition}

\begin{proof}
  The kernel of $\sigma$ induces a finite \'etale cover $X' \lra X$. After replacing $X$ by $X'$ we may assume that $\sigma$ is in fact trivial. Then $\PN(T_{X}) \simeq \PN_{n-1} \times X$ and $T_X \simeq L^{\oplus n}$ for some line bundle $L \in \Pic(X)$. Then $K_X \simeq (L^*)^{\otimes n}$.

Consider one of the many inclusions $L^* \lra \Omega_X^1$. A well known result of Bogomolov (\cite{Bog}, p. 501) says $\kappa(L^*) \le 1$. Then $\kappa(K_X) \le 1$. In fact Bogomolov's result gives more in our situation (see also \cite{Pet}, 2.19.):

\begin{lemma} \label{LBDO}
  Let $X$ be a projective manifold of dimension $n \ge 2$. Let $K_X$ be nef. Assume there exists an inclusion of sheaves
    \[0 \lra L^* \lra \Omega_X^1\]
  where $L \in \Pic(X)$ and $rL^* \equiv K_X$ for some rational $r \in \KQ^+$. Then $\kappa(X) \le \nu(X) \le 1$, i.e., $K_X^2.H^{n-2} = 0$ for every $H \in \Pic(X)$ ample.
\end{lemma}

\begin{proof}
As $K_X$ is nef, $\kappa(X) \le \nu(X)$ (\cite{MP}, IV, 2.3.). If $X$ is a surface, $\kappa(L^*) \le 1$ by (\cite{Bog}). Then $L^*$ is not big so $K_X^2 = 0$. In the case $n>2$ let $H \in \Pic(X)$ be very ample, $S = H_1 \cap \dots \cap H_{n-2}$ with general $H_i \in |H|$. Consider the composition $L^*|_S \lra  \Omega_X^1|_S \lra \Omega_S^1$. It is non--zero for otherwise we obtain a non trivial map $L^* \lra N^*_{S/X}$ and a non-zero section of $N_{S/X}^* \otimes L$ which is impossible as $N_{S/X} \otimes L^*$ is ample.  Again by (\cite{Bog}) $L^*|_S$ is not big. Then $K_X|_S$ is not big and $(K_X|_S)^2 = K_X^2.H^{n-2}=0$. Then $\nu(X) \le 1$. 
\end{proof}

We continue the proof of \Ref{Finite}. Assume $K_X \not\equiv 0$. Let $H \in \Pic(X)$ be ample. Then $L.H^{n-1}<0$ as $L^*$ is nef. By \Ref{LBDO} we have $c_1(X)^2.H^{n-2}=0$. By \Ref{ChernKl} also $c_2(X).H^{n-2}= \frac{n-1}{2n}K_X^2.H^{n-1}=0$. By \cite{Simp1}, 9.7., the universal cover $\tilde{X} \simeq \Sieg_1 \times \cdots \times \Sieg_1$. By \cite{Sieg}, $K_X$ is ample. This contradicts \Ref{KE2}. Therefore, $K_X \equiv 0$ and $X$ is almost abelian by \Ref{KE2}.
\end{proof}

\section{Infinite case}
\setcounter{equation}{0}
Assume from now on that the representation 
 \[\sigma: \pi_1(X) \lra PGl_{n-1}(\KC)\]
describing $\PN(T_X)$ has an infinite image. This is in fact not possible as we will see at the very end. The main result of this section is:

\begin{theorem} \label{Abundance}
  Let $X$ satisfy $(*)$ and assume $\sigma$ is infinite. Then $X$ is a good minimal model and $n > \kappa(X)>0$. The general fiber of the Iitaka fibration 
    \[f: X \lra Y\]
is almost abelian.
\end{theorem}

A rational map $\xymatrix{X \ar@{..>}[r] & W}$ is almost holomorphic if there exist Zariski dense open subsets $X^0 \subset X$ and $W^0 \subset W$ such that $X^0 \lra W^0$ is holomorphic and  proper. \Ref{Abundance} will follow from:

\begin{proposition} \label{Fibr}
  Let $X$ satisfy $(*)$ and assume $\sigma$ is infinite. After replacing $X$ by a finite \'etale cover, we find an almost holomorphic dominant rational map
\begin{equation} \label{VVVV}
  g: \xymatrix{X \ar@{..>}[r] & W}
\end{equation}
onto a point or a smooth $W$ of general type, $\dim X > \dim W \ge 0$. If $X_w$ denotes a very general fiber, then $X_w$ is a good minimal model. Either
\begin{enumerate}
 \item $X_w$ is almost abelian or
 \item the Iitaka fibration $X_w \lra C$ is an almost abelian fibration over a curve $C$.
\end{enumerate}
\end{proposition}

\Ref{Fibr} implies \Ref{Abundance}:

\begin{proof}[Proof of \Ref{Abundance}.]
  Replace $X$ by the finite cover from \Ref{Fibr}. Denote by $g_1: X_1 \lra W$ a resolution of the rational map $g$ in \Ref{Fibr}. Fibers are good minimal models. By \cite{Kaw85}, $C_{n,m}$ is true for $g_1$. Hence
  \begin{equation} \label{Cnm}
    \kappa(X) = \kappa(X_1) \ge \kappa(X_w) + \kappa(W) = \kappa(X_w) + \dim W.
  \end{equation}
The right hand side is either $\dim W$ or $\dim W+1$ by \Ref{Fibr}. In the case $\kappa(X_w) + \dim W = 0$, $W$ is a point. Then $X = X_w$ almost abelian and $\sigma$ is finite. This shows $\kappa(X) > 0$. By \Ref{KE2}, $n > \kappa(X)$.

Denote the rational Iitaka fibration by $f: \xymatrix{X \ar@{..>}[r] & Y}$ where $\dim Y = \kappa(X)$. Consider the two cases  $\kappa(X_w) = 0$ and $= 1$.

If $\kappa(X_w) = 0$, then $\dim Y \ge \dim W$ by \Ref{Cnm}.  The fiber $X_w$ is almost abelian and $K_X|_{X_w} \equiv 0$ by the adjunction formula. Then $f(X_w)$ will be a point and therefore $\dim Y \le \dim W$. Then $W$ and $Y$ are birational. 

If $\kappa(X_w) = 1$, then $\dim Y \ge \dim W+1$ by \Ref{Cnm}. Here $f$ will contract the general fiber $F_c$ of the Iitaka fibration of $X_w$.  Then $\dim Y \le \dim W + 1$ and we have a rational map from $Y$ to $W$ of relative dimension $1$.

General fibers of $f$ are hence $X_w$ and $F_c$ respectively. Then the general fiber of $f$ has a good minimal model. By \cite{Lai}, Theorem 4.4., $X$ has a good minimal model. By \cite{Lai}, Proposition 2.4. $X$ is a good minimal model as $K_X$ is nef. 
\end{proof}

\begin{proof}[Proof of \Ref{Fibr}]
  Denote the connected component of the identity of the Zariski closure of $\sigma(\pi_1(X))$ by $N$. Replacing $X$ by a finite \'etale cover we may assume $\sigma: \pi_1(X) \lra N$. Let $Rad(N)$ be the solvable radical of $N$. 

{\em 1. Case $N$ unsolvable.} Then $G := N/Rad(N)$ is a semisimple Lie group. By composition we obtain $\bar{\sigma}: \pi_1(X) \lra G$. The image is Zariski dense. By \cite{Ko}, 3.5. and 4.1., there exists $g: \xymatrix{X \ar@{..>}[r] & W}$ almost holomorphic to a smooth $W$ with the following property: if $Z \subset X$ passes through a very general point, then $Z$ will be contracted by $g$ if and only if
  \[\bar{\sigma}(Im(\pi_1(Z^{norm}) \lra \pi_1(X))) \subset G\]
is finite. Here $W$ is a smooth model of $Sh_{kern \bar{\sigma}}(X)$. Denote by $\pi_1: X_1 \lra X$ a resolution such that $g_1: X_1 \lra W$ is holomorphic. 

After replacing $X$ by a finite \'etale cover we may instead of finiteness assume $\pi_1^*\bar{\sigma} = g_1^*\rho$ for some {\em big} $\rho: \pi_1(W) \lra G$ (\cite{Zuo}, explanations after Theorem 1). We claim that $W$ is of general type. In the case $G$ almost simple this is \cite{Zuo}, Theorem 1. The general case easily follows by induction on the number of almost simple almost direct factors of $G$ and $C_{n,m}$ from \cite{Kaw85}. 

By \Ref{KE2} $\dim W < \dim X$. Denote by $X_w$ a very general fiber of $g$. By construction $\pi_1(X_w) \lra \pi_1(X) \lra N/Rad(N)$ is trivial, i.e.
  \begin{equation} \label{solv}
     \sigma|_{X_w}: \pi_1(X_w) \lra Rad(N) \subset N.
  \end{equation}

{\em 2. Case $N$ solvable.} Then $\sigma: \pi_1(X) \lra Rad(N)$. We allow $W$ to be a point and treat the two cases simultanously from now on.

\

\noindent {\bf Notation.} A smooth projective variety $F$ satifies condition $(F)$ if
\begin{enumerate} 
 \item there exists a map $\nu: F \lra X$, finite \'etale onto its smooth image
 \item $\sigma|_F: \pi_1(F) \lra Rad(N)$
 \item $N_{F/X} := \nu^*T_X/T_F$ admits a filtration into subbundles
  \begin{equation} \label{FiltN}
     0=N_0 \subset N_1 \subset N_2 \subset \cdots \subset N_r = N_{F/X}
  \end{equation}
such that each successive quotient $N_i/N_{i-1} \simeq \O_F^{\oplus r_i}$ is trivial. 
\end{enumerate}
Condition 3.) has to be ignored if $F \lra X$ is surjective. Note that this can happen only in case $N = Rad(N)$ solvable. In the unsolvable case, $X_w$ and every finite \'etale cover of $X_w$ satisfies $(F)$. The filtration in \Ref{FiltN} comes from the fact that we will have to consider a sequence of albanese fibrations (i.e., Shafarevich maps) in the end coming from $Rad(N)$.

\

Let $F$ satisfy $(F)$. By Lie's theorem, we find a basis of $\KC^n$ such that matrices in $Sl_n(\KC)$ representing elements in $Rad(N)$ have upper triangular form. For each $\gamma \in \pi_1(F)$ we find a unique matrix $A_{\gamma} \in Sl_n(\KC)$ of the form
   \[A_{\gamma} = \left(\begin{array}{cccc}
                         a_{11} & a_{12} & \cdots & a_{1n} \\
                          0 & a_{22} & \cdots & a_{2n} \\
                         \cdots & \cdots & \cdots & \cdots \\
                          0 & 0 & 0 & 1
                        \end{array}\right),\]
such that $\sigma(\gamma) = \PN(A_{\gamma})$. Let $\rho(\gamma) := A_{\gamma}$. This is a representation of $\pi_1(F)$ that lifts $\sigma$ to $Sl_n(\KC)$. Denote by $E$ the flat vector bundle on $F$ associated to $\rho$. Then $\PN(E) \simeq \PN(T_X)$, so 
 \[T_X \simeq E \otimes L\]
for some $L \in \Pic(F)$. The bundle $E$ admits a filtration
    \begin{equation} \label{FiltE}
      0 = E_0 \subset E_1 \subset \cdots \subset E_n = E
    \end{equation}
into flat vector bundles $E_i$ on $F$ of rank $rk E_i = i$ and flat line bundles $E_i/E_{i-1}$, $E/E_{n-1} \simeq \O_F$. Flat line bundles are nef and extensions of nef bundles are nef (\cite{DPS}). Hence each $E_i$ is numerically flat. As $\det E \equiv 0$ and $\det N_{F/X} \equiv 0$ we find
   \begin{equation} \label{LKF}
      L^{\otimes n} \equiv -K_F \equiv -\nu^*K_X.
   \end{equation}

\begin{lemma} \label{Fdescr}
  Let $F$ satisfy $(F)$. Then $K_F$ is nef, $0 \le \kappa(F) \le 1$ and $\nu(F) \le 1$.
   \begin{enumerate}
    \item If $K_F \equiv 0$, then $F$ is almost abelian.
    \item If $\kappa(F) = 1$, then $F$ is a good minimal model; the Iitaka fibration is almost abelian fibration onto a curve $C$.
   \end{enumerate}
\end{lemma}

\begin{proof}
   By \Ref{RatCurve}, $K_F$ is nef. We first show $h^0(F, \nu^*\Omega^1_X) \not= 0$. If $F \lra X$ is surjective, $N=Rad(N)$ is solvable. This gives a non trivial abelian representation of $\pi_1(F)$. Then $q(F) = h^0(F, \nu^*\Omega_X^1) \not= 0$. In the non--surjective case we have
   \begin{equation} \label{FinM}
     0 \lra N_{F/X}^* \lra \nu^*\Omega_X^1 \lra \Omega_F^1 \lra 0.
   \end{equation}
  By \Ref{FiltN} $h^0(F, N_{F/X}^*) \not= 0$. Hence $h^0(F, \nu^*\Omega^1_X) \not= 0$.

 Next consider
   \[0 \lra (E_i/E_{i-1})^*\otimes L^* \lra E_i^*\otimes L^* \lra E_{i-1}^* \otimes L^* \lra 0.\]
For $i=n$ we have $E_n^*\otimes L^* \simeq \nu^*\Omega_X^1$ and this bundle has a section. The filtration shows by induction: $h^0(F, P_1 \otimes L^*) \not= 0$ for some numerically flat $P_1 \in \Pic(F)$. Combined with \Ref{LKF} we find $h^0(F, P_2 \otimes K_F^{\otimes k}) \not= 0$ for some $k \in \KN$ and some $P_2 \in \Pic^0(F)$. By \cite{CH}, Theorem 3.2.
 \[V^m(K_F) = \{P \in \Pic^0(F) | h^0(F, K_F^{\otimes m} \otimes P)>0\},\]
if non--empty, is a finite union of torsion translates of abelian subvarieties of $\Pic^0(F)$. Then we find a torsion line bundle $P_3$ such that $H^0(F, P_3 \otimes K_F^{\otimes k}) \not= 0$. Then $\kappa(F) \ge 0$.

1.) Suppose $K_F \equiv 0$. As $\nu^*K_X \equiv K_F$ we find $c_i(\nu^*T_X) = 0$ for $i \ge 0$ by \Ref{ChernKl}. Then \Ref{FinM} and flatness of $N_{F/X}$ shows $c_i(F) = 0$ for $i \ge 0$. Then $F$ is almost abelian by \Ref{KE}.

2.) Suppose $K_F \not\equiv 0$. The above sequence for $i=n$ reads 
   \[0 \lra L^* \lra \nu^*\Omega_X^1 \simeq E^*\otimes L^* \lra E_{n-1}^*\otimes L^* \lra 0\]
By composition we obtain a map $L^* \lra \Omega_{F}^1$. If it is zero, then we obtain a non--zero section of $L \otimes N^*_{F/X}$. The filtration \Ref{FiltN} gives $h^0(F, L) \not= 0$. Then \Ref{LKF} implies $K_F \equiv 0$, a contradiction. Hence $L^* \lra \Omega_F^1$ is non-zero. By \Ref{LBDO} we have $\kappa(F) \le \nu(F) \le 1$. In the case $\kappa(F) = 1$, $K_F$ is abundant (\cite{MP}, IV, 2.5.). The general fiber $F_c$ of the Iitaka fibration $F \lra C$ satisfies $(F)$ with numerically trivial canonical divisor. Then $F_c$ is almost abelian by 1.).
\end{proof}

This next Lemma completes the proof of \Ref{Fibr}.\end{proof}

\begin{lemma} \label{KodNull}
   Let $F$ satisfy $(F)$. If $\kappa(F) = 0$, then $F$ is almost abelian.
\end{lemma} 

\begin{proof}
   By \Ref{Fdescr} we have to prove that $F$ is a good minimal model. Let $R_1 := Rad(N)$, $R_2 := [R_1, R_1]$, $R_3 := [R_2, R_2]$, etc. Let $t \in \KN$ be minimal such that $R_t = \{id\}$. We have $t=1$ iff $Rad(N) = \{id\}$.

Choose $i$ minimal such that $\sigma(Im(\pi_1(F) \lra \pi_1(X))) \subset R_i$. If $i < t$, then we have a non trivial abelian representation $\pi_1(F) \lra R_i/R_{i+1}$. Then $q(F) > 0$. By \cite{KawAb}, $\kappa(F)=0$ implies that $alb_F: F \lra Alb(F)$ is a fibration, i.e., surjective with connected fibers. A general fiber $F_a$ again satisfies $(F)$. By \cite{Lai}, 4.2., it suffices to show that $F_a$ has a good minimal model. Then $F$ has a good minimal model and $K_F$ nef implies $F$ is a good minimal model (\cite{Lai}, 2.4.).

{\em Claim.} $F_a$ admits a finite \'etale cover $F'_a \lra F_a$ with $\kappa(F'_a) = 0$ and  $\sigma|_{F'_a}: \pi_1(F') \lra R_{i+1}$.

By \Ref{Fdescr}, $0 \le \kappa(F_a) \le 1$. Assume $\kappa(F_a) = 1$. Then $F_a$ is a good minimal model by \Ref{Fdescr}. Then $C_{n,m}$ is true (\cite{Kaw85}), so $0 = \kappa(F) \ge \kappa(F_a) + \kappa(Alb(F)) = 1$, a contradiction. Therefore $\kappa(F_a) = 0$. 

Let $H := Im(\pi_1(F_a) \lra \pi_1(F))$. As $F_a$ is contracted by $alb_F$,
   \[H \cap [\pi_1(F), \pi_1(F)] \subset H\]
is of finite index. Then $F_a$ admits a finite \'etale cover $F'_a$ such that $Im(\pi_1(F'_a) \lra \pi_1(F)) \subset  [\pi_1(F), \pi_1(F)]$. Then $\sigma|_{F'_a}: \pi_1(F'_a) \lra R_{i+1}$. This proves the claim.

Consider $F' := F'_a$. In the case  $i+1 < t$ proceed as above: consider the albanese of $F'$ and a general fiber. By induction we see that it suffices to prove \Ref{KodNull} under the additional assumption $\sigma|_F$ trivial.

As $\sigma$ is infinite on $X$ but trivial on $F$, $\nu: F \lra X$ will not be surjective. The conormal bundle $N^*_{F/X}$ has a section. Then $h^0(F, \nu^*\Omega^1_X) \not= 0$. Triviality of $\sigma$ implies $\nu^*T_X \simeq L^{\oplus n}$. Then $h^0(F, \nu^*\Omega^1_X) \ge n$. Filtration \Ref{FiltN} shows $h^0(N^*_{F/X}) \le rk N_{F/X}$. Then $q(F)= h^0(F, \Omega_F^1) \ge \dim F$. Then $F \lra Alb(F)$ is an isomorphism by \cite{KawAb}.
\end{proof}

\section{Abelian group schemes}
\setcounter{equation}{0}

A map between projective manifolds $f: X \lra Y$ is a {\em smooth abelian group scheme}, if $f$ is smooth, every fiber is an abelian variety and if $f$ admits a smooth section.

\begin{theorem} \label{AbSch}
  Let $X$ satisfy $(*)$ and assume $\sigma$ is infinite. Then there exists a finite \'etale cover $X' \lra X$ such that $X'$ is a smooth abelian group scheme over a base $Y$ such that $K_Y$ is ample ($\dim X > \dim Y > 0$).
\end{theorem}

\begin{proof}
  By \Ref{Abundance}, $X$ is a good minimal model and the Iitaka fibration $f: X \lra Y$ is an almost abelian fibration onto $Y$ normal, $0 < \dim Y = \kappa(X) < \dim X$. By \cite{Kaw}, Theorem 2, $f$ is equidimensional. Denote a general fiber by $X_y$. Denote by $\nu: A_y \lra X_y$ the finite \'etale cover from an abelian $A_y$. 

Let $Z \subset A_{y}$ be an irreducible positive dimensional subvariety passing through a very general point. If
   \begin{equation} \label{ImFundGr}
     Im(\pi_1(Z^{norm}) \lra \pi_1(X))
   \end{equation}
is infinite for every such $Z$, then $X$ has generically large fundamental group on $X_y$ in the sense of Koll\'ar (\cite{Ko}, 6.1.). Then by \cite{Ko}, 6.3., some finite \'etale cover of $X$ is birational to an abelian group scheme $G \lra S$.

\begin{lemma}
  If $X$ as above has generically large fundamental group on $X_y$, then \Ref{AbSch} is true.
\end{lemma}

\begin{proof}
 By Koll\'ar's result we can replace $X$ by a finite \'etale cover to obtain
  \[\xymatrix{G \ar[d] \ar@{..>}[r] & X \ar[d] \\
               S \ar@{..>}[r]  & Y.}\]
The upper rational map must be holomorphic, as $X$ is free of rational curves. The map $G \lra S$ has a smooth section. Then the lower map must be holomorphic, too. As $G \lra X$ is birational and $G \lra S$ smooth, every fiber of $X \lra Y$ is reduced. Denote by $Y' \subset X$ the image of the section of $G \lra S$. Then $Y'$ intersects every fiber transversally in a single point. Then $Y'$ and $Y$ are smooth. Fibers of $X \lra Y$ are irreducible. As there are no rational curves, every fiber is normal. Then $X \lra Y$ is smooth. 

By \cite{Ko}, (5.9.) we have $\kappa(X) = \kappa(Y)$. Then $Y$ is of general type. As $X \lra Y$ admits a section, $K_Y$ is nef. The base point theorem implies $|dK_Y|$ is spanned, $d \gg 0$.  Exceptional fibers of the Iitaka fibration of $Y$ are again covered by rational curves. Then \Ref{RatCurve} and the existence of a section implies $K_Y$ is ample.
\end{proof}

The next Proposition finishes the proof of \Ref{AbSch}. \end{proof}

\begin{proposition} \label{GenLFG}
  $X$ as in \Ref{AbSch} has generically large fundamental group on the general fiber $X_y$.
\end{proposition}

\begin{proof}
Assume \Ref{ImFundGr} is finite for some $Z$. The largest dimensional such $Z \subset A_y$ are subtori $B_y$ of $A_y$ and translates. Indeed, if $H = kern(\pi_1(A_y) \lra \pi_1(X_1))$, then these are the fibers of $\xymatrix{A_y \ar@{..>}[r] & Sh_H(A_y)}$. Let $\xymatrix{X \ar@{..>}[r] & T}$ be the relative Shafarevich map for $X \lra Y$ (\cite{Ko}, 3.10.). After blowing up $X$ we obtain a diagram
 \[\xymatrix{ X_1 \ar[r]^{\pi} \ar[d]^{g} & X \ar[d]^f \\
              T \ar[r]^{\tau} & Y}\]
where the general fiber of $g$ is covered by the abelian variety $B_y$ from above. Let $\tilde{C} = H_1 \cap \cdots \cap H_{n-1} \subset X_1$ be a general curve where $H_i \in |mH|$, $m \gg 0$, for some ample $H \in \Pic(X_1)$. Our aim is to prove
  \begin{equation} \label{numtriv}
    \pi^*K_X.(mH)^{n-1} = \pi^*K_X.\tilde{C} = 0.
  \end{equation}
As $K_X$ is nef, this implies $K_X \equiv 0$. Then $X$ is almost abelian by \Ref{KE2} and $\sigma$ is finite. A contradiction. 

Let $C = g(\tilde{C}) \subset T$ and $F := g^{-1}(C)$. We may assume $F$ and $C$ are smooth, $g_F=g|_F: F \lra C$ is an almost abelian fibration, $\tilde{C}$ is a ramified multisection. The map $\pi_F=\pi|_F: F \lra X$ is an embedding near the general fiber of $g_F$. If we can show $\pi_F^*K_X \equiv 0$, then \Ref{numtriv} is true.

Let $\xymatrix{F' \ar@{..>}[r] & F}$ be a dominant rational and generically finite map from some smooth projective $F'$. The induced map $\pi_{F'}: \xymatrix{F' \ar@{..>}[r] & X}$ is holomorphic \Ref{RatCurve}. A resolution diagram shows that $\pi_{F'}^*K_X \equiv 0$ implies $\pi_F^*K_X \equiv 0$. We may therefore replace $F$ by $F'$. For example, after replacing $F$ by a semistable reduction, we may assume $g_F: F \lra C$ semistable. Our aim is to find a much simpler model:

\begin{lemma} \label{Prod}
  Let $\beta: B \times D \lra X$ be a generically finite map where $B$ is an abelian variety of dimension $\dim B > 0$ and $D$ is a curve. Then $\beta^*K_X \equiv 0$.
\end{lemma}

\begin{proof}
     Let $\tilde{D}$ be a general curve in $B \times D$, i.e., the intersection of $\dim(B)$ general hyperplanes. The map
  \[\beta^*\Omega_X^1|_{\tilde{D}} \lra \Omega^1_{B \times D}|_{\tilde{D}} \simeq pr_1^*\O_B^{\dim B} \oplus pr_2^*K_D|_{\tilde{D}}\]
is generically surjective. Projection to the first summand yields a generically surjective map $\alpha^*\Omega_X^1|_{\tilde{D}} \lra \O_{\tilde{D}}^{\oplus \dim B}$. Denote the image by $E$. Then $E$ is a rank $\dim B >0$ vector bundle on $\tilde{D}$. The cokernel of $E \hookrightarrow \O_{\tilde{D}}^{\oplus \dim B}$ is torsion on $\tilde{D}$, implying $\deg E \le 0$. By condition $(*)$
  \[\frac{\beta^*K_X.\tilde{D}}{n} = \frac{\deg \beta^*\Omega_X^1|_{\tilde{D}}}{n} \le \frac{\deg E}{\dim B} \le 0.\]
Then $\beta^*K_X.\tilde{D} \le 0$. But $K_X$ is nef. Then $\beta^*K_X.\tilde{D} = 0$.
\end{proof}

\

We continue the proof of \Ref{GenLFG}. Choose $U \subset C$ Zariski open and dense such that $g$ is smooth over $U$ and such that $F^0 = g^{-1}(U)$ embeds into $X$ via $\pi_F$. Consider the relative tangent sequence
   \begin{equation} \label{RelTangSeq}
     0 \lra T_{F^0/U} \lra T_{F^0} \lra g_F^*T_U \lra 0
   \end{equation}

{\em Claim.} The Kodaira Spencer map $\theta: T_U \lra R^1g_{F,*}T_{F^0/U}$ is zero. 

It suffices to show that \Ref{RelTangSeq} splits holomorphically when restricted to a fiber $F_u$ of $F^0 \lra U$. Restricted to $F_u$ we have
   \begin{equation} \label{TangSeq}
    0 \lra T_{F_u} \lra T_F|_{F_u} \lra \O_{F_u} \lra 0.
  \end{equation}
We find an abelian variety $B_u$ and a finite \'etale map 
   \begin{equation} \label{B}
   \mu: B_u \lra F_u, \quad \mbox{s.t.} \;\; Im(\pi_1(B_u) \lra \pi_1(X)) = \{id\}.
  \end{equation}
The pull back of \Ref{TangSeq} to $B_u$ gives 
  \[0 \lra T_{B_u} \simeq \O_{B_u}^{\oplus \dim B_u} \lra \mu^*T_F \lra \O_{B_u} \lra 0\]
and it suffices to show that this sequence splits holomorphically. This means we have to shows $\mu^*T_F \simeq \O_{B_u}^{\oplus \dim F}$. The last sequence shows $\det \mu^*T_F \simeq \O_{B_u}$.

Near $F_u$, $\pi_F$ is an embedding. Then we have a sequence of vector bundles
  \begin{equation} \label{NSurdr}
     0 \lra \mu^*T_F \lra \mu^*\pi_F^*T_X \lra \mu^*N \lra 0.
  \end{equation}
By \Ref{B}, $\mu^*\pi_F^*T_X$ is projectively trivial, hence $\mu^*\pi_F^*T_X \simeq L^{\oplus n}$ for some $L \in \Pic(B_u)$. As $\mu^*\pi_F^*K_X \equiv 0$ by the adjunction formula, $L \equiv 0$. The inclusion $T_{B_u} \hookrightarrow 
\mu^*\pi_F^*T_X$ shows that $L$ has a section. Then $L \simeq \O_{B_u}$ and $\mu^*\pi_F^*T_X \simeq \O_{B_u}^{\oplus n}$.

Then \Ref{NSurdr} implies $\mu^*N$ is globally generated and $\det \mu^*N \simeq \det \mu^*T_F \simeq \O_{B_u}$. A globally generated vector bundle $E$ with $c_1(E)=0$ is trivial (indeed, $E$ and $E^*$ are nef. By \cite{DPS}, 1.16., for any $s \in H^0(E)$ we have $Z(s) = \emptyset$. This gives an exact sequence of vector bundles, now use \cite{DPS}, 1.15., and induction). Hence $\mu^*\pi^*N \simeq \O_{B_u}^{\oplus n - \dim F}$. Dually we obtain in the same way $\mu^*T_{F} \simeq  \O_{B_u}^{\dim F}$. This proves our claim.

The vanishing of the Kodaira Spencer implies $g$ is isotrivial over $U$, i.e., any two fibers are isomorphic. Then there exists a ramified base change $D \lra C$, such that $F \times_C D$ is birational to $B \times D$ where $B$ is some smooth abelian variety. The induced rational map $\beta: \xymatrix{B \times D \ar@{..>}[r] & X}$ is holomorphic by \Ref{RatCurve} and generically finite. By \Ref{Prod}, $\beta^*K_X \equiv 0$. This proves \Ref{numtriv}.
\end{proof}

\section{Proof of the Main Theorem}
\setcounter{equation}{0}

\begin{theorem} \label{Arakelov}
  Let $X$ satisfy $(*)$. Then $X$ is almost abelian.
\end{theorem}

\begin{proof}
 If $\sigma: \pi_1(X) \lra PGl_{n-1}(\KC)$ has a finite image, then we are done by \Ref{Finite}. Assume $Im(\sigma)$ is infinite.

By \Ref{AbSch}, after replacing $X$ by the finite \'etale cover, $X$ is an abelian group scheme over a smooth base $Y$ such that $K_Y$ is ample, $d := \dim Y > 0$. Let $E^{1,0} = f_*\Omega^1_{X/Y}$. Then $f^*E^{1,0} \simeq \Omega^1_{X/Y}$ and we have the exact sequence
    \[0 \lra f^*\Omega_Y^1 \lra \Omega_X^1 \lra f^*E^{1,0} \lra 0\]
Restrict everything to a smooth section which we denote by $Y$ as well. Consider the polarization given by $K_Y$. Condition $(*)$ implies
   \[\mu_{K_Y}(\Omega_X^1|_Y) \le \mu_{K_Y}(E^{1,0}).\]
Then $d\mu_{K_Y}(\Omega_Y^1) + (n-d)\mu_{K_Y}(E^{1,0}) = n\mu_{K_Y}(\Omega_X^1) \le n\mu_{K_Y}(E^{1,0})$ or $\mu_{K_Y}(\Omega_Y^1) \le \mu_{K_Y}(E^{1,0})$. By \cite{VZ2}, Theorem~1 and Remark~2, on the other hand 
 $\mu_{K_Y}(R^1f_*\KC) \le \mu_{K_Y}(\Omega_Y^1)$ where $\mu_{K_Y}(R^1f_*\KC) = 2\mu(E^{1,0})$. Then we find $2\mu_{K_Y}(\Omega_Y^1) \le \mu_{K_Y}(\Omega_Y^1)$. Then $K_Y^d = 0$, a contradiction.
\end{proof}

\end{document}